\documentclass[12pt,leqno,twoside]{amsart}
\usepackage{amssymb,amsmath,amsthm,soul,color}
\usepackage{t1enc}
\usepackage[cp1250]{inputenc}
\usepackage{a4,indentfirst,latexsym}
\usepackage{graphics}
\usepackage{mathrsfs}
\usepackage{cite,enumitem,graphicx}
\usepackage[colorlinks=true,urlcolor=blue,
citecolor=red,linkcolor=blue,linktocpage,pdfpagelabels,
bookmarksnumbered,bookmarksopen]{hyperref}
\usepackage[english]{babel}
\usepackage[left=2.61cm,right=2.61cm,top=2.72cm,bottom=2.72cm]{geometry}
\usepackage[metapost]{mfpic}
\usepackage[colorinlistoftodos]{todonotes}
\usepackage[normalem]{ulem}

\newtheorem{Th}{Theorem}[section]
\newtheorem{Prop}[Th]{Proposition}
\newtheorem{Lem}[Th]{Lemma}
\newtheorem{Cor}[Th]{Corollary}

\newenvironment{altproof}[1]
{\noindent
	{\em Proof of {#1}}.}
{\nopagebreak\mbox{}\hfill $\Box$\par\addvspace{0.5cm}}

\newcommand{\wt}{\widetilde}

\newcommand{\vp}{\varphi}

\newcommand{\eps}{\varepsilon}

\def\div{\mathop{\mathrm{div}}}

\def\Z{\mathbb{Z}}

\def\R{\mathbb{R}}

\def\curl{\mathrm{curl}}
\def\dim{\mathrm{dim}}

\def\V{\mathcal{V}}
\def\E{\mathcal{E}}
\def\J{\mathcal{J}}

\def\W{\mathcal{W}}
\def\D{\mathcal{D}}

\def\tJ{\wt{\cJ}}

\def\rh{\rightharpoonup}

\newcommand{\cB}{{\mathcal B}}
\newcommand{\cC}{{\mathcal C}}
\newcommand{\cD}{{\mathcal D}}
\newcommand{\cE}{{\mathcal E}}

\newcommand{\cH}{{\mathcal H}}
\newcommand{\cI}{{\mathcal I}}
\newcommand{\cJ}{{\mathcal J}}
\newcommand{\cK}{{\mathcal K}}
\newcommand{\cL}{{\mathcal L}}
\newcommand{\cM}{{\mathcal M}}
\newcommand{\cN}{{\mathcal N}}
\newcommand{\cO}{{\mathcal O}}
\newcommand{\cP}{{\mathcal P}}

\newcommand{\cT}{{\mathcal T}}

\newcommand{\cV}{{\mathcal V}}
\newcommand{\cW}{{\mathcal W}}

\renewcommand{\dim}{{\rm dim}\,}

\newcommand{\al}{\alpha}

\newcommand{\ga}{\gamma}

\newcommand{\Ga}{\Gamma}
\newcommand{\Om}{\Omega}

\newcommand{\rr}{\R^3}

\newcommand{\mbS}{\mathbb{S}}

\def\curlop{\nabla\times}
\newcommand{\weakto}{\rightharpoonup}
\newcommand{\pa}{\partial}

\newcommand{\tX}{\widetilde{X}}
\newcommand{\tu}{\widetilde{u}}
\newcommand{\tv}{\widetilde{v}}

\newcommand{\cTto}{\stackrel{\cT}{\longrightarrow}}

\numberwithin{equation}{section}

\title[Nonlinear Curl-curl problems]{Nonlinear curl-curl problems in $\mathbb{R}^3$}

\author[J. Mederski]{Jaros\l aw Mederski}
\author[J. Schino]{Jacopo Schino}

\address[J. Mederski and J. Schino]{\newline\indent
	{\newline\indent 
		Institute of Mathematics,
		\newline\indent 
		Polish Academy of Sciences,
		\newline\indent 
		ul. \'Sniadeckich 8, 00-656 Warsaw, Poland}
}
\email{\href{mailto:jmederski@impan.pl}{jmederski@impan.pl}}
\email{\href{mailto:jschino@impan.pl}{jschino@impan.pl}}


\begin{document}
	\begin{abstract} We survey recent results concerning  ground states and bound states  $u\colon\mathbb{R}^3\to\mathbb{R}^3$ to the  curl-curl problem
		$$\nabla\times(\nabla\times u)+V(x)u= f(x,u) \qquad\hbox{ in } \mathbb{R}^3,$$
		which originates from the nonlinear Maxwell equations. The energy functional associated with this problem is strongly indefinite due to the infinite dimensional kernel of $\nabla\times(\nabla\times \cdot)$. The growth of the nonlinearity $f$ is superlinear and subcritical at infinity or purely critical and we demonstrate a variational approach to the problem involving the generalized Nehari manifold. We also present some refinements of known results.
	\end{abstract}
	
	\maketitle

	\section*{Introduction}
	\setcounter{section}{1}
	
	We look for weak solutions to the {\em semilinear  curl-curl problem}
	\begin{equation}\label{eq}
		\nabla\times(\nabla\times u) +V(x)u= f(x,u), \qquad x\in\R^3,
	\end{equation}
	originating from the {\em Maxwell equations} in the differential form
	\begin{equation}\label{e-Maxwell}
		\begin{cases}
			\nabla\times\cH=\cJ+\partial_t\cD \, & \text{(Amp\`ere's Law)}\\
			\nabla\cdot\cD=\rho \, & \text{(Gauss's Electric Law)}\\
			\nabla\times\cE=-\partial_t\cB \, & \text{(Faraday's Law)}\\
			\nabla\cdot\cB=0 \, & \text{(Gauss's Magnetic Law)},
		\end{cases}
	\end{equation}
	where $\cH,\cJ,\cD,\cE,\cB\colon\rr\times\R\to\rr$ are time-dependent vector fields and $\rho\colon\rr\times\R\to\R$ is the electric charge density. In particular, $\cH$ is the magnetic intensity field, $\cJ$ the electric current intensity, $\cD$ the electric displacement field, $\cE$ the electric field, and $\cB$ the magnetic induction. We consider as well the {\em constitutive relations}
	\begin{equation}\label{e-const}
		\begin{cases}
			\cD=\epsilon\cE+\cP\\
			\cH=\frac1\mu\cB-\cM,
		\end{cases}
	\end{equation}
	where $\cP,\cM\colon\rr\times\R\to\rr$ are, respectively, the polarization field (which depends on $\cE$, in general nonlinearly) and the magnetization field, while $\epsilon,\mu\colon\rr\to\R$ are, respectively, the permittivity and the permeability of the material. \\
	\indent In order to derive \eqref{eq} we make additional assumptions about the physical model. We begin by considering absence of electric charges ($\rho=0$), electric currents ($\cJ=0$), and magnetization ($\cM=0$).  Then, plugging \eqref{e-const} into \eqref{e-Maxwell} and differentiating with respect to the time variable we obtain the {\em electromagnetic wave equation}
	\begin{equation}\label{eqem}
	\nabla\times\left(\frac1\mu\nabla\times\cE\right)+\epsilon\partial_t^2\cE=-\partial_t^2\cP.
	\end{equation}
	Moreover, we assume that $\cE$ and $\cP$ are \textit{monochromatic waves}, i.e., $\cE(x,t)=\cos(\omega t)u(x)$ and $\cP(x,t)=\cos(\omega t)P(x)$ for some $\omega\in\R$ and $u,P\colon\rr\to\rr$, which leads to the curl-curl problem \eqref{eq}, where  $\mu\equiv1$, $V(x)=-\eps(x)\omega^2$, and $f=\omega^2P$ models a nonlinear polarization in the medium, see \cite{Mederski2015,Stuart91,Stuart93} and the references therein.
	Note that solving \eqref{eqem}, or \eqref{eq} in the monochromatic setting, one obtains the electric displacement field according to the first constitutive relation in \eqref{e-const},	
	the magnetic induction is given by integrating Faraday’s law in time, and the magnetic field is given by the second constitutive relation. Since there are no currents, nor charges, then the Gauss laws are satisfied: $\div(\cD)=\div(\cB)=0$.  Thus, we find the {\em exact propagation} of the electromagnetic field according to the Maxwell equations.
	\\
	\indent
	Another motivation has been provided by Benci and Fortunato \cite{BenFor}, who introduced a model for a unified field theory for classical electrodynamics based on a semilinear perturbation of the Maxwell equations in the spirit of the Born-Infeld theory \cite{BornInfeld}. In the magnetostatic case, in which the electric field vanishes and the magnetic field is independent of time, this  leads to an equation of the form \eqref{eq} with $u$ replaced with $A$, the gauge potential related to the magnetic field.\\
	\indent A major mathematical difficulty of \eqref{eq} and similar curl-curl problems is that the differential operator $u\mapsto\nabla\times\nabla\times u$ has an infinite-dimensional kernel, i.e., the space of gradient vector fields; this makes the associated energy functional
	\begin{equation}\label{eq:action}
		\E(u)=\frac12\int_{\R^3}|\curlop u|^2\,dx +V(x)|u|^2 - \int_{\R^3} F(x,u)\,dx,
	\end{equation}
	where $f=\partial_u F$,
	strongly indefinite, i.e., unbounded from above and below (when $F\ge0$) even on subspaces of finite codimension and such that its critical points have infinite Morse index; for instance, this is the case in a model example
	\begin{equation}\label{eq:modelEx}
		f(x,u)=\Gamma(x)\min\{|u|^{p-2},|u|^{q-2}\}u\quad\hbox{ with }2<p\leq 6\leq q
	\end{equation}
	where $\Gamma\in L^{\infty}(\R^3)$ is $\Z^3$-periodic, positive, and bounded away from $0$. Another issue is that the Fr\'echet differential of the energy functional is not sequentially weak-to-weak* continuous, therefore the limit point of a weakly convergent sequence needs not be a critical point of the functional. Moreover, one has to struggle with the lack of compactness because the problem is set in the whole space $\rr$.\\
	\indent We underline that the aforementioned difficulties in dealing with curl-curl problems  have given rise to several simplifications in the literature. The most widely used is the scalar or vector nonlinear Schr\"odinger equation, where, e.g., one assumes that the term $\nabla(\div u)$ in $\nabla\times\nabla\times u=\nabla(\div u)-\Delta u$ is negligible and can therefore be removed from the equation, or uses the so-called \textit{slowly varying envelope approximation}. Nevertheless, such approximations may produce non-physical solutions, which do not describe the \textit{exact} propagation of electromagnetic waves in Maxwell's equations, as remarked, e.g., in \cite{AAS-C,CCDY}, whence the importance of curl-curl problems from a physical point of view.\\
	\indent  As far as we know, the first papers dealing with exact solutions to Maxwell's equations are \cite{McStTr,Stuart91}, where the electromagnetic wave equation \eqref{eqem} is turned in an ODE and treated with ad hoc techniques. The same approach is used in the series of papers \cite{Stuart93,Stuart04,StuartZhou96,StuartZhou01,StuartZhou03,StuartZhou05,StuartZhou10}.\\ 
	\indent In general, the curl-curl problem \eqref{eq} seems to be difficult to study by considering radial solutions and the consequently associated ODE, since it does not admit any radial solutions for a large class of $V$ and $F$.
  In fact, it follows from \cite[Theorem 1.1]{BDPR:2016} that any $\cO(3)$-equivariant (weak) solution to \eqref{eq} is trivial.\\
	\indent The curl-curl problem \eqref{eq} in $\R^3$ has been solved for the first time by Azzollini, Benci, D'Aprile, and Fortunato in \cite{BenForAzzAprile} in the cylindrically symmetric setting. If $V=0$ and $F(x,u)$ depends only on $|u|$ as in \cite{BenForAzzAprile}, then one can restrict the considerations to
	the fields of the form
	\begin{equation}\label{eq:sym1}
		u(x)=\al(r,x_3)\begin{pmatrix}-x_2\\x_1\\0\end{pmatrix},\qquad r=\sqrt{x_1^2+x_2^2},
	\end{equation}
	which are divergence-free, so $\curlop(\curlop u)=-\Delta u$ and one can study \eqref{eq} by means of standard variational methods (however, there may still exist solutions which are not of this form). Other results in the cylindrically symmetric setting have been obtained in \cite{DAprileSiciliano,BDPR:2016,HirschReichel,Zeng,Mandel,GaczkowskiMS}. Let us mention that the solutions found in \cite{DAprileSiciliano} are orthogonal to the subspace of vector fields of the form \eqref{eq:sym1}. Clearly the cylindrically symmetric approach is not applicable if $V$ or $f$ in \eqref{eq} lack this symmetry or, even when \eqref{eq} does preserve this symmetry, if we look for ground state solutions, i.e., nontrivial solutions with minimal energy. The natural question whether ground state solutions must have some symmetry properties arises, but so far it is an open problem.
\\
	\indent Similar difficulties have appeared also in curl-curl problems on bounded domains \cite{BartschMederski,BartschMederskiJFA}, where  Bartsch and the first author
	investigated a problem similar to \eqref{eq} and paired with boundary conditions that model the case of a medium surrounded by a perfect conductor (i.e., the electric field on the boundary of the medium is tangential to it), obtaining the system
	\begin{equation*}\label{e-bddcurl}
		\begin{cases}
			\nabla\times\nabla\times u+V(x)u=f(x,u) & \, \text{in } \Omega\\
			\nu\times u=0 & \, \text{on } \partial\Omega,
		\end{cases}
	\end{equation*}
	with $\nu\colon\partial\Omega\to\mbS^2$ the outer normal unit vector. As in the linear case, e.g. \cite{BuffaAmmariNed,Leis68,KirschHettlich,Monk},
	or in \cite{BenFor,DAprileSiciliano}, the authors split the function space they work with into a divergence-free part, i.e., $\div (V(x)u)=0$, and a curl-free part, which allows to build a more tractable variational setting of the problem. Then they adopt techniques from \cite{SzulkinWeth} that exploit a generalization of the Nehari manifold, which needs not be of class $\cC^1$ (also known in the literature as the Nehari--Pankov manifold, cf. \cite{Pankov}). However, under some technical assumptions about $F$, the generalized Nehari manifold is proved to be homeomorphic to the unit sphere in the subspace of the divergence-free vector fields. In addition, since such a manifold is a natural constraint, by a suitable minimization argument the authors find a ground state solution, as well as infinitely many solutions with the energy converging to infinity. The advantage of working in a bounded domain is that, despite the presence of the subspace of curl-free vector fields, which does not embed compactly in any Lebesgue (or Orlicz) function space, a variant of the Palais--Smale condition is satisfied, which provides some compactness in the aforementioned minimization argument.
	Other approaches have been developed in subsequent work concerning the the Brezis-Nirenberg problem \cite{BrezisNirenberg} for the curl-curl operator  \cite{MederskiJFA2018,MSz}; see also the survey \cite{BartschMederskiJFTA}.\\ 
	\indent Back to \eqref{eq}, where no variants of the Palais--Smale condition are available because of the infinite measure of $\R^3$, a careful concentration-compactness analysis on a suitable generalized Nehari manifold $\cN_\E$ (see \eqref{DefOfNehari1} for the definition) has been demonstrated in \cite{Mederski2015}, which seems to be the first work on ground state solutions of \eqref{eq} in the nonsymmetric setting with $V\leq 0$ and where \eqref{eq:modelEx} with $2<p<6<q$ is the model in mind. In the current work, we refine results of \cite{Mederski2015}, present the variational approach, and provide a simpler argument resolving the compactness issue inspired by the recent work by Szulkin and the authors \cite{MSchSz}.\\ 
	\indent In \cite{MSchSz}, the multiplicity problem of bound states to \eqref{eq} with $V=0$ has been considered together with a large class of nonlinearities which have supercritical growth at $0$ and subcritical growth at infinity; this is in the spirit of the zero mass case of Berestycki and Lions \cite{BerLions}, see condition (N2) below. However, as shown by the examples below, we admit nonlinearities which are more general than in \eqref{eq:modelEx} with $2<p<6<q$, and this requires a new functional setting for \eqref{eq} as well as a new critical point theory. The reason is that the methods based on the constraint $\cN_{\E}$ cannot be applied straightforwardly here since $\cN_{\E}$ may not be homeomorphic to the unit sphere in the subspace of divergence-free vector fields as in \cite{BartschMederski,Mederski2015}.  Note that although $\cE$ has the classical linking geometry, the well-known linking results, e.g. of Benci and Rabinowitz \cite{BenciRabinowitz}, are not applicable due to the lack of weak-to-weak$^*$ continuity of $\cE'$.\\
	\indent In order to state the main results we assume that
	the growth of  $f$ is controlled by an $N$-function $\Phi:\R\to [0,\infty)$ of class $\cC^1$ such that
	\begin{itemize}
		\item[(N1)] $\Phi$ satisfies the $\Delta_2$- and the $\nabla_2$-condition globally.
	\end{itemize}
	$N$-functions and condition (N1) will be introduced in the next section and are  standard in the theory of Orlicz spaces \cite{RaoRen}. For the subcritical nonlinearities we shall need the following growth conditions:
	\begin{itemize}
		\item[(N2)] $\displaystyle\lim_{t\to 0}\frac{\Phi(t)}{t^6}=\lim_{t\to\infty}\frac{\Phi(t)}{t^6}=0$,
		\item[(N3)] $\displaystyle\lim_{t\to\infty}\frac{\Phi(t)}{t^2}=\infty$,
	\end{itemize}
	where (N2) is inspired by \cite{BerLions} and (N2), (N3) describe supercritical behaviour at $0$ and superquadratic but subcritical at infinity.
	Now we collect our assumptions on the nonli\-nearity $F(x,u)$.
	\begin{itemize}
		\item[(F1)] $F\colon\R^3\times\R^3\to\R$ is differentiable with respect to the second variable $u\in\R^3$ for a.e. $x\in\R^3$ and $f=\pa_uF\colon\R^3\times\R^3\to\R^3$ is a Carath\'eodory function (i.e., measurable in $x\in\R^3$, continuous in $u\in\R^3$ for a.e.\ $x\in\R^3$). Moreover, $f$ is $\Z^3$-periodic in $x$, i.e., $f(x,u)=f(x+y,u)$ for all $u\in \R^3$ and almost all $x\in\R^3$ and $y\in\Z^3$.
		\item[(F2)] There are $c_1$, $c_2>0$ such that
		$$|f(x,u)|\le c_1\Phi'(|u|)\text{ and }F(x,u)\ge c_2\Phi(|u|)$$
		for every $u\in\R^3$ and a.e. $x\in\R^3$, where $\Phi$ satisfies (N1)--(N3).
		\item[(F3)] For every $u\in\R^3$ and a.e. $x\in\R^3$ 
		$$\langle f(x,u),u\rangle\geq 2F(x,u).$$
		\item[(F4)] If $ \langle f(x,u),v\rangle = \langle f(x,v),u\rangle >0$, then
		$\ \displaystyle F(x,u) - F(x,v)
		\le \frac{\langle f(x,u),u\rangle^2-\langle f(x,u),v\rangle^2}{2\langle f(x,u),u\rangle}$.
	\end{itemize}
	
	We say that $F$ is {\em uniformly strictly convex with respect to $u\in\R^3$} if and only if for any compact $A\subset(\R^3\times\R^3)\setminus\{(u,u):\;u\in\R^3\}$
	$$
	\inf_{\genfrac{}{}{0pt}{}{x\in\R^3}{(u_1,u_2)\in A}}
	\left(\frac12\big(F(x,u_1)+F(x,u_2)\big)-F\left(x,\frac{u_1+u_2}{2}\right)\right) > 0.
	$$
	
	We provide some examples. First we note that if $G = G(x,t)\colon \R^3\times\R\to\R$ is differentiable with respect to $t$, $g:=\partial_tG$ is a Carath\'eodory function, $G(x,0)=0$, $M\in GL(3)$ is an invertible $3\times 3$ matrix, and
	\begin{equation} \label{increasing}
		F(x,u) = G(x,|Mu|) \quad  \text{and} \quad t\mapsto g(x,t)/t \text{ is non-decreasing for } t>0,
	\end{equation}
	then $F$ satisfies (F3) (cf. \cite{SzulkinWeth}) and it is easy to see that (F4) holds. Note that \eqref{increasing} implies $g(x,0)=0$, so $f$ is continuous also at $u=0$.\\
	\indent Suppose  $\Ga\in L^\infty(\R^3)$ is $\Z^3$-periodic, positive, and bounded away from $0$. Take 
	$$F(x,u):=\Gamma(x)W(|Mu|^2)$$ where $W$ is a function of class $\cC^1$, $W(0)=W'(0)=0$, and $t\mapsto W'(t)$ is non-decreasing on $(0,\infty)$. Then we check that (F1)--(F4) are satisfied (here $G(x,t)=\Gamma(x)W(t^2)$, so \eqref{increasing} holds). If
	\begin{equation}\label{Expq1}
		W(t^2)=\frac1p\big((1+|t|^q)^{\frac{p}{q}}-1\big)
	\end{equation}
	or
	\begin{equation}\label{Expq2}
		W(t^2)=\min\Big\{\frac1p|t|^p+\frac1q-\frac1p,\frac1q|t|^q\Big\}
	\end{equation}
	with $2<p<6<q$, then we can take $\Phi(t)=W(t^2)$ and we see that (F2) holds as well. Note that
	if $W'(t)$ is constant on some interval $[a,b]\subset (0,\infty)$, then 
	\begin{equation} \label{equality}
		0<\displaystyle F(x,u) - F(x,v)
		=\frac{\langle f(x,u),u\rangle^2-\langle f(x,u),v\rangle^2}{2\langle f(x,u),u\rangle}
	\end{equation}
	for $a<|v|<|u|<b$ and a stronger variant of (F4), i.e., \cite[(F4)]{Mederski2015}, is no longer satisfied, so we cannot apply  variational techniques relying on minimization on the Nehari-Pankov manifold $\cN_{\cE}$ (defined in \eqref{DefOfNehari1}) as in \cite{SzulkinWeth,Mederski2015}. 
	Moreover, we can consider $f$ that cannot be controlled by any $N$-function associated with $L^p(\R^3,\R^3)+L^q(\R^3,\R^3)$ for $2<p<6<q$ as in \cite{Mederski2015}. Indeed, let us consider
	$W(t^2)=\frac12(|t|^2-1)\ln (1+|t|)-\frac14|t|^2+\frac12|t|$ for $|t|\geq 1$, $W(t^2)=\frac{\ln2}q(|t|^q-1)+\frac14$ for $|t|<1$, then 
	\begin{equation*}
		f(x,u)=\begin{cases}
			\Gamma(x)\ln(1+|u|)u  & \text{ if } |u|\geq 1,\\
			\Gamma(x)\ln(2)|u|^{q-2}u & \text{ if } |u|<1
		\end{cases}
	\end{equation*}
	and note that (F1)--(F4) are satisfied.\\
	\indent As our final example, we take $F(x,u)=\Gamma(x) \Phi(|u|)$ where $\Phi(0)=0$, 
	\begin{equation*}
		\Phi'(t) =\begin{cases}
			t^5/(1-\ln t)  & \text{ if } t\le 1,\\
			t & \text{ if } 1\le t\le 2, \\
			at^5/\ln t & \text{ if } t\ge2,
		\end{cases}
	\end{equation*}
	and $a=2^{-4}\ln2$. Obviously, $F$ satisfies \eqref{increasing} and hence (F3) and (F4), and \eqref{equality} holds for $1<|u|<2$. It is easy to see that (F1)--(F2) and (N1)--(N3) hold (to check (N1) it is convenient to use Lemma \ref{AllProp}). Note that $\lim_{t\to 0} \Phi(t)/t^6 = 0$, but $\lim_{t\to 0} \Phi(t)/t^q=\infty$ for any $q>6$; similarly $\Phi(t)/t^6\to 0$ but $\Phi(t)/t^p\to\infty$ as $t\to\infty$ for any $p<6$. Note also that in the last two examples we can replace $|u|$ with $|Mu|$. 
	\\
	\indent Let $S$ be the classical Sobolev constant of the embedding $\cD^{1,2}(\R^3)$ into $L^6(\R^3)$ and $\Psi=\Phi^*$ the complementary function to $\Phi$ (cf. Section \ref{sec:varsetting}). We assume that
	\begin{itemize}
		\item[(V)] $\Phi$ satisfies the $\nabla'$-condition globally, $\Phi\circ\Psi^{-1}$ is convex and satisfies the $\nabla_2$-condition globally, $V\in L^{(\Phi\circ \Psi^{-1})^* \circ \Psi}(\R^3)$, $V(x)<0$ for a.e. $x\in\R^3$, and $|V|_{3/2}<S$.
	\end{itemize}
	Since, as we will see, $N$-functions are even, when we write $\Psi^{-1}$ we mean $(\Psi|_{[0,\infty)})^{-1}$, while the $\nabla'$-condition will be introduced in the next section.
	We remark that $\Phi\circ\Psi^{-1}$ is in fact an $N$-function and that (N2) and $V\in L^{(\Phi\circ \Psi^{-1})^* \circ \Psi}(\R^3)$ imply $V\in L^{3/2}(\R^3)$, see Lemma \ref{lem:V}.\\ 
\indent Let $\cD(\curl,\Phi)$ be the space of functions  $u$ such that $\nabla\times u$ is square integrable and $u$ is in the Orlicz space $L^\Phi(\R^3,\R^3)$; see the next section for a more accurate definition. Then $\mathcal{E}\in \cC^1(\cD(\curl,\Phi), \R)$ and  critical points of $\cE$ are weak solutions to \eqref{eq}.\\
\indent Our first aim is to present the following result.
\begin{Th}\label{ThMain} Assume that (F1)--(F4) hold. Then: \\
	(a) If (V) holds and $F$ is convex with respect to $u\in\R^3$, or $V=0$ and $F$ is uniformly strictly convex with respect to $u\in\R^3$, then equation \eqref{eq} has a ground state solution, i.e., there is a critical point $u\in\mathcal{N}_{\cE}$ of $\E$ such that
	$$\E(u)=\inf_{\mathcal{N}_{\cE}}\E>0,$$
	where
	\begin{eqnarray}\label{DefOfNehari1}
	\mathcal{N}_{\cE} &:=& \{u\in \D(\curl,\Phi): u\neq 0,\; 
	\E'(u)[u]=0,\\\nonumber
	&&\hbox{ and }\E'(u)[\nabla\vp]=0\,\hbox{ for any }\vp\in \cC_0^{\infty}(\R^3)\}.
	\end{eqnarray}
	(b) If $V=0$, $F$ is uniformly strictly convex with respect to $u\in\R^3$, and $F$ is even in $u$,  then there is an infinite sequence $(u_n)\subset\cN_{\cE}$ of geometrically distinct solutions of  \eqref{eq}, i.e., solutions such that $(\Z^3\ast u_n)\cap (\Z^3\ast u_m)=\emptyset$ for $n\neq m$, where
	$$\Z^3\ast u_n:=\{u_n(\cdot+y): y\in\Z^3\}.$$
\end{Th}

Theorem \ref{ThMain} has been obtained in \cite{MSchSz} in case $V=0$. If $\Phi'(t)=\min\{|t|^{p-2},|t|^{q-2}\}t$ with $2<p<6<q$, then $(V)$ holds provided that $V\in L^{\frac{p}{p-2}}(\R^3)\cap L^{\frac{q}{q-2}}(\R^3)$, $V<0$ a.e. on $\R^3$, and $|V|_{\frac{3}{2}}<S$, therefore Theorem \ref{ThMain} (a) generalizes \cite[Theorem 2.1]{Mederski2015} and we can consider nonlinearities like \eqref{Expq2} and \eqref{Expq1}.\\
\indent The critical problem in $\R^3$, i.e., $V=0$ and $f(x,u)=|u|^4u$, has been investigated by Szulkin and the first author in \cite{MSz}. Recall that if $\Phi(t)=\frac16|t|^6$, then $W^6_0(\curl;\R^3):=\D(\curl,\Phi)$ and $W^6_0(\curl;\R^3)=W^6(\curl;\R^3)$, where
$$
W^6(\curl;\R^3) := \big\{u\in L^6(\R^3,\R^3): \curlop u\in L^2(\R^3,\R^3)\big\}
$$
is endowed with the norm
$
\|u\|_{W^6(\curl;\R^3)} := \left(|u|^2_6+|\curlop u|^2_2\right)^{1/2}
$. Denote the kernel of $\curlop (\cdot )$ in $W^6(\curl;\R^3)$ by
$$\cW:=\{w\in W^6(\curl;\R^3): \curlop w=0\}$$
and let $S_{\curl}$ be the largest possible constant such that the inequality
\begin{equation}\label{eq:neq}
	\int_{\R^3}|\curlop u|^2\, dx\geq S_{\curl}\inf_{w\in \cW}\Big(\int_{\R^3}|u+w|^6\,dx\Big)^{\frac13}
\end{equation}
holds for every $u\in W^6(\curl;\R^3)\setminus \W$. Inequality \eqref{eq:neq} is in fact (trivially) satisfied also for $u\in \cW$ because then both sides are zero. According to \cite{MSz}, we present the following result.

\begin{Th}\label{Th:main1}
$S_\curl> S$, $\inf_{\cN_{\cE}}\cE=\frac13S_{\curl}^{3/2}$, and it is attained by a ground state solution to \eqref{eq} with $V=0$ and $f(x,u)=|u|^4u$.
\end{Th}

Multiple entire solutions in the Sobolev-critical case ($V=0$ and $f(x,u)=|u|^4u$) are obtained, for the first time, by Gaczkowski and the authors in \cite{GaczkowskiMS} combining the symmetry introduced in \cite{BenForAzzAprile} of the form \eqref{eq:sym1} with another introduced by Ding in \cite{Ding}, which restores compactness in the critical case. They also extend rigorously an equivalence result, known for the classical formulations, that relates the weak solutions to \eqref{eq} with the weak solutions to Schr\"odinger equations with singular potentials, see \cite{GaczkowskiMS} for details.\\
\indent 
The paper is organized as follows. In Section \ref{sec:varsetting}, we recall the properties of $N$-functions and Orlicz spaces and we prove some new results required by the presence of the potential $V$ and the assumption $(V)$. In Section \ref{sec:criticaslpoitth}, we recall the critical point theory for strongly indefinite functionals based on  \cite{MSchSz}, which  also solves the problem of multiplicity of bound states. In Section \ref{sec:Jcurlcurl}, we prove some preliminary results about the energy functional for the curl-curl problem, showing the abstract theory from Section \ref{sec:criticaslpoitth} suits this concrete context. Finally, Sections \ref{sec:T1} and \ref{sec:T2} are where we prove Theorems \ref{ThMain} and \ref{Th:main1} respectively.

\section{Preliminaries and variational setting}\label{sec:varsetting}

Here and in the sequel, $|\cdot|_q$ denotes the $L^q$-norm.\\
\indent Now, following \cite{RaoRen}, we recall some basic definitions and results about $N$-functions and Orlicz spaces.
A function $\Phi\colon\R\to[0,\infty)$ is called an {\em $N$-function} if and only if it is even, convex, and satisfies
$$\Phi(t)=0\Leftrightarrow t=0,\quad\lim_{t\to 0}\frac{\Phi(t)}{t}=0,\quad\hbox{and }\lim_{t\to\infty}\frac{\Phi(t)}{t}=\infty.$$
Given an $N$-function $\Phi$, we can associate with it another function $\Phi^*\colon\R\to[0,\infty)$ defined by
\[
\Phi^*(t) := \sup\{s|t|-\Phi(s):s\ge 0\}
\]
which is an $N$-function as well. $\Phi^*$ is called the {\em complementary function} to $\Phi$ while $(\Phi,\Phi^*)$ is called a {\em complementary pair} of $N$-functions. To simplify the notations, we will write $\Psi$ for $\Phi^*$. Note that $\Psi^*=\Phi$.\\
\indent
We also recall from \cite[Section II.3]{RaoRen} that $\Phi$ satisfies the $\Delta_2$-{\em condition globally} (denoted $\Phi\in\Delta_2$) if there exists $K>1$ such that for every $t\in\R$
\[
\Phi(2t)\le K\Phi(t)
\]
while $\Phi$ satisfies the $\nabla_2${\em -condition globally} (denoted $\Phi\in\nabla_2$) if there exists $K'>1$ such that for every $t\in\R$
\[
\Phi(K't)\ge 2K'\Phi(t).
\]
Similarly, $\Phi$ satisfies the $\Delta'${\em -condition globally} (denoted $\Phi\in\Delta'$) if there exists $c>0$ such that for every $s,t\in\R$
\[
\Phi(st) \le c \Phi(s) \Phi(t),
\]
while $\Phi$ satisfies the $\nabla'${\em -condition globally} (denoted $\Phi\in\nabla'$) if there exists $c'>0$ such that for every $s,t\in\R$
\[
\Phi(s) \Phi(t) \le c'\Phi(st).
\]
In order to make the text more fluent, {\em from now on, when we say that the $\Delta_2$-, $\nabla_2$-, $\Delta'$-, or $\nabla'$-condition holds, we mean that it holds globally}.\\
\indent The set 
$$L^\Phi:=L^\Phi(\R^3,\R^3):=\Big\{u\colon\R^3\to\R^3\text{ measurable } : \int_{\R^3}\Phi(\alpha|u|)\,dx<\infty \text{ for some } \alpha>0\Big\}$$ is a vector space and it is called an Orlicz space; if $\Phi\in\Delta_2$, then one can take the equivalent definition
\[
L^\Phi=\Big\{u\colon\R^3\to\R^3\text{ measurable } : \int_{\R^3}\Phi(|u|)\,dx<\infty\Big\}.
\]
Moreover, the space $L^\Phi$ becomes a Banach space (cf. \cite[Theorem III.2.3, Theorem III.3.10]{RaoRen}) if endowed with the norm
$$|u|_{\Phi} := \inf\Big\{k>0 : \int_{\R^3}\Phi\Big(\frac{|u|}{k}\Big) \, dx \le 1\Big\}.$$
We can define an equivalent norm on $L^\Phi$ by letting
$$|u|_{\Phi,1}:=\sup\Big\{\int_{\R^3}|u|\, |u'|\, dx: \int_{\R^3}\Psi(|u'|)\, dx \leq 1,\; u'\in L^{\Psi}\Big\},$$
see \cite[Proposition III.3.4]{RaoRen} (note that in \cite{RaoRen} these results are formulated for the space $\cL^\Phi$; however, no distinction needs to be made between $\cL^\Phi$ and $L^\Phi$, see the comment following  \cite[Corollary III.3.12]{RaoRen}). 
Finally, if both $\Phi$ and $\Psi$ satisfy the $\Delta_2$-condition, then $L^\Phi$ is reflexive and $L^\Psi$ is its dual \cite[Corollary IV.2.9 and Theorem IV.2.10]{RaoRen}. Similarly, for any measurable $\Om\subset\R^3$ one can define $$L^\Phi(\Om):=\Big\{\xi\colon\Om\to\R\text{ measurable and} \int_{\Om}\Phi(\alpha|\xi|)<\infty \text{ for some } \alpha>0\Big\}$$
and endow it with the norm $|\cdot|_\Phi$ defined as above. \\
\indent Recall that $L^\Phi=L^\Phi(\R^3)^3$ can be identified \cite[Lemma 2.1]{MSchSz}.\\
\indent Before going on, for the reader's convenience we recall some important facts.
\begin{Lem}\label{AllProp}$\mbox{}$
\begin{itemize}
\item [(i)] The following are equivalent:
	\begin{itemize}
	\item $\Phi\in\Delta_2$;
	\item there exists $K>1$ such that $t\Phi'(t)\le K\Phi(t)$ for every $t\in\R$;
	\item there exists $K'>1$ such that $t\Psi'(t)\ge K'\Psi(t)$ for every $t\in\R$;
	\item $\Psi\in\nabla_2$.
	\end{itemize}
\item [(ii)] $\Phi\in\nabla'$ if and only if $\Psi\in\Delta'$.
\item [(iii)] For every $u\in L^\Phi$, $u'\in L^\Psi$ there holds
$$\int_{\R^3}|u|\,|u'|\,dx\le\min\{|u|_{\Phi,1}|u'|_{\Psi},|u|_{\Phi}|u'|_{\Psi,1}\}.$$
\item [(iv)] Let $u_n$, $u\in L^\Phi$. Then $|u_n-u|_\Phi\to 0$ implies that $\int_{\R^3}\Phi(|u_n-u|)\,dx\to 0$. If $\Phi\in\Delta_2$, then $\int_{\R^3}\Phi(|u_n-u|)\,dx\to 0$ implies $|u_n-u|_\Phi\to 0$.
\item [(v)] Let $X\subset L^\Phi$ and suppose $\Phi\in\Delta_2$. Then $X$ is bounded if and only if $\{\int_{\R^3}\Phi(|u|)\,dx:E\in X\}$ is bounded.
\end{itemize}
\end{Lem}
\begin{proof}
Point (ii) is due to \cite[Theorem II.3.11]{RaoRen}, while the remaining ones have been proved in \cite[Lemma 2.2]{MSchSz}.
\end{proof}

Now we prove a preliminary result that we will use in some of the next lemmas.

\begin{Lem}\label{lem:conj}
Let $1<p<\infty$ and define $p':=\frac{p}{p-1}$. Let $\Phi$ be any $N$-function and define $\Psi:=\Phi^*$. Then
\[
\lim_{t\to 0}\frac{\Phi(t)}{|t|^p} = 0 \Leftrightarrow \lim_{t\to 0}\frac{\Psi(t)}{|t|^{p'}} = \infty \quad \text{and} \quad \lim_{t\to \infty}\frac{\Phi(t)}{|t|^p} = 0 \Leftrightarrow \lim_{t\to \infty}\frac{\Psi(t)}{|t|^{p'}} = \infty.
\]
\end{Lem}
\begin{proof}
Suppose that $\lim_{t\to 0}\frac{\Phi(t)}{|t|^p} = 0$. For $\eps>0$, there exists $\delta>0$ such that $\Phi(t)/|t|^p \le \eps$ if $|t|\le\delta$. We have
\[
\frac{\Psi(t)}{|t|^{p'}} \ge \sup\left\{ \frac{s}{|t|^{p'-1}} - \frac{\Phi(s)}{|t|^{p'}} : s\in[0,\delta] \right\} \ge \frac{1}{|t|^{p'-1}} \sup\left\{ s - \eps\frac{s^p}{|t|} : s\in[0,\delta] \right\}
\]
and the maximizer of 
$s - \eps s^p/|t|$ when $s\ge0$ is $\bar s=(t/p\eps)^{p'-1}$. Since we are considering the limit as $t\to0$, we can take $t$ so small that $\bar s\in[0,\delta]$ and so
\[
\frac{\Psi(t)}{|t|^{p'}} \ge \left(\frac{1}{p^{p'-1}} - \frac{1}{p^{p'}}\right)\frac{1}{\eps^{p'-1}} \quad \text{if } t \ll 1,
\]
hence $\lim_{t\to 0}\Psi(t)/|t|^{p'} =\infty$. We argue similarly if $\lim_{t\to \infty}\frac{\Phi(t)}{|t|^p} = 0$.\\
\indent Suppose that $\lim_{t\to 0}\frac{\Psi(t)}{|t|^{p'}} = \infty$. For $\eps>0$ there exist $M>\delta>0$ such that $\Psi(t)/|t|^{p'} > 1/\eps$ if $|t| < \delta$ and $\Psi(t)/|t| > 1/\eps$ if $|t| > M$. Observe that for $|t| \ll 1$
\[
\sup\left\{ s - \frac{\Psi(s)}{|t|} : s>M \right\} \le  \sup\left\{ s\left(1-\frac{1}{|t|\eps}\right) : s\ge0 \right\} = 0
\]
and
we have
\begin{eqnarray*}
\frac{\Phi(t)}{|t|^{p}} &\leq&  
\frac{1}{|t|^{p-1}} \sup\left\{ s - \frac{\Psi(s)}{|t|} : s\in[0,\delta) \right\} \le \frac{1}{|t|^{p-1}} \sup\left\{ s - \frac{s^{p'}}{\eps|t|} : s\ge0 \right\}\\
&=& \left(\frac{1}{(p')^{p-1}} - \frac{1}{({p'})^{p}}\right)\eps^{p-1}.
\end{eqnarray*}
There follows that $\lim_{t\to 0}\Phi(t)/|t|^p=0$. We similarly argue that $\lim_{t\to \infty}\frac{\Psi(t)}{|t|^{p'}} = \infty$ implies $\lim_{t\to \infty}\Phi(t)/|t|^p=0$.
\end{proof}

From now on we, assume (F1)--(F4), (N1), (N3), $\Phi$ will denote an $N$-function as in (F2) and $\Psi$ will denote its complementary function. We assume also
	\begin{itemize}
	\item[(N2')] There exists $c>0$ such that $\Phi(t)\leq c |t|^6$ for every $t\in\R$.
\end{itemize}
(N2) will be assumed only if it is required, so that some of the results below are valid also for the critical case $\Phi(t)=\frac16|t|^6$. Moreover, we will denote by $|\cdot|_\Phi$ any of the two (equivalent) norms defined above, unless differently required.\\
\indent Let $\D(\curl,\Phi)$ be the completion of $\mathcal{C}_0^{\infty}(\R^3,\R^3)$ with respect to the norm
$$\|u\|_{\curl,\Phi}:=\big(|\curlop u|_2^2+|u|_{\Phi}^2\big)^{1/2}.$$
 The subspace of divergence-free vector fields is defined by
 \[
 \begin{aligned}
 \V
 &:= \left\{v\in \D(\curl,\Phi):\; \int_{\R^3}\langle v,\nabla \vp\rangle\,dx=0
 \text{ for any }\vp\in \cC^\infty_0(\R^3)\right\}\\
 &= \{v\in \D(\curl,\Phi):\; \div v=0\}
 \end{aligned}
 \]
 where $\div v$ is to be understood in the distributional sense. Let $\cD:=\D^{1,2}(\R^3,\R^3)$
 be the completion of $\cC^{\infty}_0(\R^3,\R^3)$ with respect to the norm 
 $$\|u\|_{\D}:=|\nabla u|_2,$$
 and let $\W$ be the closure of
 $\big\{\nabla\vp: \vp\in \cC^{\infty}_0(\R^3)\big\}$ 
 in $L^{\Phi}$. In view of (N2') and  by Lemma \ref{AllProp} (iii),
 $L^6(\R^3,\R^3)$ is continuously embedded in $L^{\Phi}$.\\
\indent  
The following Helmholtz decomposition has been obtained in \cite{MSchSz} provided that (N2) holds, however the proof is valid for (N2'), cf. \cite{MSz}.
\begin{Lem}\label{defof} $\V$ and $\cW$ are closed subspaces of $\D(\curl,\Phi)$  and
	\begin{equation*}
	\D(\curl,\Phi)=\V\oplus \W.
	\end{equation*}
	Moreover, $\cV\subset\cD$ and the norms $\|\cdot\|_{\cD}$ and $\|\cdot\|_{curl,\Phi}$ are equivalent in $\V$.
\end{Lem}

Observe that in view of Lemma \ref{defof}, $\cV$ is continuously embedded in $L^\Phi$.\\
\indent We introduce a norm in $\V\times\W$ by the formula
\[
\|(v,w)\|:=\bigl(\|v\|_{\cD}^2+|w|_{\Phi}^2\bigr)^{\frac{1}{2}}
\]
and consider the energy functional defined by \eqref{eq:action} on $\cD(\curl,\Phi)$,
and 
\begin{equation}\label{eqJ}
\J(v,w):=\frac{1}{2}\int_{\R^3}|\nabla v|^2+V(x)|v+w|^2\,dx-\int_{\R^3}F(x,v+w)\,dx.
\end{equation}
defined on $\V\times\W$.\\
\indent 
The next lemma justifies some requirements in the condition (V).

\begin{Lem}\label{lem:V}
Assume $\Phi\circ\Psi^{-1}$ is convex.
\begin{itemize}
	\item [(i)] $\Phi\circ\Psi^{-1}$ is an $N$-function that satisfies the $\Delta_2$-condition.
	\item [(ii)] If $V\in L^{(\Phi\circ \Psi^{-1})^* \circ \Psi}(\R^3)$, then  $V\in L^{3/2}(\R^3)$.
	\item [(iii)] If, moreover, $\Phi$ satisfies the $\nabla'$-condition, then $\|Vu\|_\Psi\to0$ as $\|u\|_\Phi\to0$ and $\|Vu\|_\Psi$ is bounded if $\|u\|_\Phi$ is.
\end{itemize}
\end{Lem}
\begin{proof}
\textit{(i)} Clearly $\Upsilon:=\Phi\circ\Psi^{-1}$ is even and convex, so now we prove that $\lim_{t\to 0}\Upsilon(t)/t=0$ and $\lim_{t\to\infty}\Upsilon(t)/t=\infty$. From (N2'), (N3), and Lemma \ref{lem:conj}
\[
\liminf_{t\to 0}\frac{\Psi(t)}{|t|^{6/5}}>0 \quad \text{and} \quad \lim_{t\to\infty}\frac{\Psi(t)}{t^2}=0,
\]
which in turn implies
\[
\limsup_{t\to 0}\frac{\Psi^{-1}(t)}{|t|^{5/6}}<\infty \quad \text{and} \quad \lim_{t\to\infty}\frac{\Psi^{-1}(t)}{\sqrt{t}} = \infty
\]
and so
\[
\lim_{t\to 0}\frac{\Upsilon(t)}{t} = \lim_{t\to 0} \frac{\Phi\bigl(\Psi^{-1}(t)\bigr)}{\bigl(\Psi^{-1}(t)\bigr)^6} \frac{\bigl(\Psi^{-1}(t)\bigr)^6}{t} = 0 \quad \text{and} \quad \lim_{t\to\infty}\frac{\Upsilon(t)}{t} = \lim_{t\to 0} \frac{\Phi\bigl(\Psi^{-1}(t)\bigr)}{\bigl(\Psi^{-1}(t)\bigr)^2} \frac{\bigl(\Psi^{-1}(t)\bigr)^2}{t} = \infty.
\]
We exploit Lemma \ref{AllProp} (i) to prove that $\Upsilon\in\Delta_2$. Since $\Phi\in\Delta_2$, for a.e. $t\in\R$
\[
t\Upsilon'(t) = \frac{t\Phi'\bigl(\Psi^{-1}(t)\bigr)}{\Psi'\bigl(\Psi^{-1}(t)\bigr)} = \frac{t\Psi^{-1}(t)\Phi'\bigl(\Psi^{-1}(t)\bigr)}{\Psi^{-1}(t)\Psi'\bigl(\Psi^{-1}(t)\bigr)} \le C\frac{t\Phi\bigl(\Psi^{-1}(t)\bigr)}{\Psi\bigl(\Psi^{-1}(t)\bigr)} = C\Upsilon(t)
\]
where $C>0$ only depends on $\Phi$ and can be supposed to be greater than $1$.\\
\indent
\textit{(ii)} From (N2') and Lemma \ref{lem:conj} we have
$$
\liminf_{t\to 0}\frac{\Psi(t)}{t^{6/5}} > 0 \quad \text{and} \quad \liminf_{t\to \infty}\frac{\Psi(t)}{t^{6/5}} > 0,
$$
hence
$$
\limsup_{t\to0} \frac{\Psi^{-1}(t)}{t^{5/6}} < \infty \quad \text{and} \quad \limsup_{t\to\infty} \frac{\Psi^{-1}(t)}{t^{5/6}} < \infty.
$$
This, again with (N2'), yields
$$
\limsup_{t\to0} \frac{\Upsilon(t)}{t^5} < \infty \quad \text{and} \quad \limsup_{t\to\infty} \frac{\Upsilon(t)}{t^5} < \infty,
$$
and so, still via Lemma \ref{lem:conj},
$$
\liminf_{t\to 0}\frac{\Upsilon^*(t)}{t^{5/4}} > 0 \quad \text{and} \quad \liminf_{t\to \infty}\frac{\Upsilon^*(t)}{t^{5/4}} > 0,
$$
hence there exist $C_1,C_2>0$ such that for all $t\in\R$
\[
\Psi^{-1}(t) \le C_1 |t|^{5/6} \quad \text{and} \quad |t|^{5/4} \le C_2 \Upsilon^*(t).
\]
Then
\[\begin{split}
\int_{\R^3} |V(x)|^{3/2} \, dx & = \alpha^{-3/2} \int_{\R^3} \big|\Psi^{-1}\bigl(\Psi\bigl(\alpha V(x)\bigr)\bigr)\big|^{3/2} \, dx \le (C_1/\alpha)^{3/2} \int_{\R^3} |\Psi\bigl(\alpha V(x)\bigr)|^{5/4} \, dx\\
& = (C_1/\alpha)^{3/2} C_2\int_{\R^3} \Upsilon^*\bigl(\Psi\bigl(\alpha V(x)\bigr)\bigr) \, dx < \infty,
\end{split}\]
where $\alpha$ is from the condition $V\in L^{\Upsilon^* \circ \Psi}(\R^3)$.\\
\indent
\textit{(iii)} In view of Lemma \ref{AllProp} it suffices to prove that $\int_{\R^3}\Psi(\alpha|Vu|)\,dx\to0$ as $\|u\|_\Phi\to0$ and that $\int_{\R^3}\Psi(\alpha|Vu|)\,dx$ is bounded if $\|u\|_\Phi$ is, where $\alpha$ is as before. If $\bar c>0$ is the constant associated with the $\Delta'$-condition of $\Psi$, then
\[
\int_{\R^3}\Psi(\alpha|Vu|)\,dx \le \bar{c}\int_{\R^3} \Psi(\alpha V) \Psi(|u|) \, dx \le \bar{c} \|\Psi(\alpha V)\|_{\Upsilon^*,1} \|\Psi(|u|)\|_\Upsilon
\]
and $\|\Psi(\alpha V)\|_{\Upsilon^*,1}<\infty$ because $\int_{\R^3} \Upsilon^*(\Psi(\alpha V)) \, dx < \infty$, hence we only need to prove that $\|\Psi(|u|)\|_\Upsilon\to0$ as $\|u\|_\Phi\to0$; but this is obvious in view of Lemma \ref{AllProp} (iv) and item (i) in this lemma because
\[
\int_{\R^3} \Upsilon\bigl(\Psi(|u|)\bigr) \, dx = \int_{\R^3} \Phi(|u|) \, dx \to 0.
\]
Likewise for the boundedness.
\end{proof}

The next lemma is an improvement of \cite[Lemma 2.5]{MSchSz} provided in \cite{SchinoPhD}, which does not require $\Phi$ to be strictly convex.

\begin{Lem}\label{lem:phd}
There exists $C>0$ such that for all $t\in\R$
\[
\Psi(\Phi'(t)) \le C\Phi(t).
\]
\end{Lem}
\begin{proof}
From (N1), Lemma \ref{AllProp} (i), and \cite[Theorem I.III.3]{RaoRen} there holds
\[
\Psi(\Phi'(t)) = t\Phi'(t) - \Phi(t) \le (K-1)\Phi(t).\qedhere
\]
\end{proof}

In virtue of Lemmas \ref{lem:V} (iii) and \ref{lem:phd}, we can prove as in \cite[Proposition 2.6]{MSchSz} that $\J$ is of class $\cC^1$. It is then standard to show the following result.
\begin{Prop}\label{PropSolutE}
Let $u=v+w\in\V\oplus\W$. Then $(v,w)$ is a critical point of $\J$ if and only if $u$ is a critical point of $\cE$ if and only if $\E$ is a weak solution to \eqref{eq}, i.e. $$\int_{\R^3}\langle u,\nabla\times\nabla\times\vp\rangle\,dx=\int_{\R^3}\langle -V(x)u+f(x,u),\vp\rangle\,dx\quad\hbox{ for any }\vp\in\cC^\infty_0(\R^3,\R^3).$$
\end{Prop}

\section{Critical point theory}\label{sec:criticaslpoitth}

We recall the abstract setting from \cite{BartschMederskiJFA,BartschMederski,MSchSz}.
Let $X$ be a reflexive Banach space with the norm $\|\cdot\|$ and a topological direct sum decomposition $X=X^+\oplus\tX$, where $X^+$ is a Hilbert space with a scalar product $\langle .\,,.\rangle$. For $u\in X$ we denote by $u^+\in X^+$ and $\tu\in\tX$ the corresponding summands so that $u = u^++\tu$. We may assume $\langle u,u \rangle = \|u\|^2$ for any
$u\in X^+$ and $\|u\|^2 = \|u^+\|^2+\|\tu \|^2$. The topology $\cT$ on $X$ is defined as the product of the norm topology in $X^+$ and the weak topology in $\tX$. Thus $u_n\cTto u$ is equivalent to $u_n^+\to u^+$ and $\tu_n\weakto\tu$.\\
\indent 
Let $\cJ$ be a functional on $X$ of the form 
\begin{equation*}
\cJ(u) = \frac12\|u^+\|^2-\cI(u) \quad\text{for $u=u^++\tu \in X^+\oplus \tX$}.
\end{equation*}
The set
\begin{equation*}
\cM := \{u\in X:\, \cJ'(u)|_{\tX}=0\}=\{u\in X:\, \cI'(u)|_{\tX}=0\}
\end{equation*}
obviously contains all critical points of $\cJ$. Suppose the following assumptions hold.
\begin{itemize}
	\item[(I1)] $\cI\in\cC^1(X,\R)$ and $\cI(u)\ge \cI(0)=0$ for any $u\in X$.
	\item[(I2)] $\cI$ is $\cT$-sequentially lower semicontinuous:
	$u_n\cTto u\quad\Longrightarrow\quad \liminf \cI(u_n)\ge \cI(u)$.
	\item[(I3)] If $u_n\cTto u$ and $\cI(u_n)\to \cI(u)$, then $u_n\to u$.
	\item[(I4)] $\|u^+\|+\cI(u)\to\infty$ as $\|u\|\to\infty$.
	\item[(I5)] If $u\in\cM$ then $\cI(u)<\cI(u+v)$ for every $v\in \tX\setminus\{0\}$.
\end{itemize}
Clearly, if a strictly convex functional $\cI$ satisfies (I4), then (I2) and (I5) hold. Note that 
for any $u\in X^+$ we find  $m(u)\in\cM$ which is the unique global maximizer of $\J|_{u+\tX}$.
Note that 
$\cM$ needs not be a differentiable manifold because $\cI'$ is only required to be continuous.
In order to apply classical critical point theory like the mountain pass theorem to $\cJ\circ m: X^+\to \R$ we need some additional assumptions.
\begin{itemize}
	\item[(I6)] There exists $r>0$ such that $a:=\inf\limits_{u\in X^+,\|u\|=r} \cJ(u)>0$.
	\item[(I7)] $\cI(t_nu_n)/t_n^2\to\infty$ if $t_n\to\infty$ and $u_n^+\to u^+\ne 0$ as $n\to\infty$.
\end{itemize}

According to \cite[Theorem 4.4]{BartschMederskiJFA}, if (I1)--(I7) hold and 
\[
c_\cM := \inf_{\ga\in\Ga}\sup_{t\in [0,1]} \cJ(\ga(t)),
\]
where
\[
\Ga := \{\ga\in\cC([0,1],\cM):\ga(0)=0,\ \|\ga(1)^+\|>r, \text{ and } \cJ(\ga(1))<0\},
\]
then  $c_{\cM}\ge a>0$ and $\cJ$ has a $(PS)_{c_\cM}$-sequence $(u_n)$ in $\cM$, i.e., $\cJ'(u_n)\to 0$ and $\cJ(u_n)\to c_{\cM}$. If, in addition,
$\cJ$ satisfies a variant of the Palais--Smale condition in $\cM$, then $c_\cM$ is achieved by a critical point of $\cJ$. Since we look for solutions to \eqref{eq} in $\R^3$ and not in a bounded domain as in \cite{BartschMederskiJFA}, this condition is no longer satisfied. We consider the set
\begin{equation*}
\cN := \{u\in X\setminus\tX: \cJ'(u)|_{\R u\oplus \tX}=0\} = \{u\in\cM\setminus\tX: \cJ'(u)[u]=0\} \subset\cM
\end{equation*}
and we require the following condition on $\cI$:
\begin{itemize}
	\item[(I8)] $\frac{t^2-1}{2}\cI'(u)[u]+\cI(u) - \cI(tu+v)=\frac{t^2-1}{2}\cI'(u)[u] + t\cI'(u)[v] + \cI(u) - \cI(tu+v) \leq 0$\\ for every $u\in \cN$, $t\ge 0$, $v\in \tX$.
\end{itemize}
In \cite{BartschMederski,BartschMederskiJFA}, it was additionally assumed that \emph{strict inequality holds} provided $u\neq tu+v$. This stronger variant of (I8)
implies that for any $u^+\in X^+\setminus\{0\}$ the functional $\cJ$ has a unique critical point $n(u^+)$ on the half-space $\R^+u^+ +\tX$. Moreover, $n(u^+)$ is the global maximizer of $\cJ$ on this half-space, the map
$$n\colon SX^+=\{u^+\in X^+: \|u^+\|=1\} \to \cN$$
is a homeomorphism, the set $\cN$ is a topological manifold, and it is enough to look for critical points of $\cJ\circ n$. $\cN$ is called the Nehari-Pankov manifold. This is the approach of \cite{SzulkinWethHandbook}. However, if the weaker condition (I8) holds, this procedure cannot be repeated. In particular, $\cN$ needs not be a manifold.\\ 
\indent
Let $u\in\cN$. In view of (I8) we get by explicit computation
\begin{equation*}\label{eq:JonN}
	\cJ(tu+v)=\cJ(tu+v)-\cJ'(u)\Big[\frac{t^2-1}{2}u+tv\Big]\leq \cJ(u)
\end{equation*}
for any $t\geq 0$ and $v\in\tX$. Hence, $u$ is a (not necessarily unique) maximizer of $\cJ$ on  $\R^+u +\tX$.\\
\indent 
Let 
\[
\tJ := \cJ\circ m\colon X^+\to\R. 
\]
Before stating the main results of this section, we recall the following properties (i)--(iv) taken from \cite[Proof of Theorem~4.4]{BartschMederskiJFA}. Note that (I8) has not been used there.
	\begin{itemize}
		\item[(i)] For each $u^+\in X^+$ there exists a unique $\tu\in \tX$ such that $m(u^+):=u^++\tu\in\cM$. This $m(u^+)$ is the minimizer of $\cI$ on $u^++\tX$.
		\item[(ii)] $m\colon X^+\to \cM$  is a homeomorphism with the inverse $\cM\ni u\mapsto u^+\in X^+$.
		\item[(iii)] $\wt{\cJ}=\cJ\circ m \in\cC^1(X^+,\R)$.
		\item[(iv)]$\wt{\cJ}'(u^+) = \cJ'(m(u^+))|_{X^+}\colon X^+\to\R$ for every $u^+\in X^+$.
	\end{itemize}

As usual, $(u_n)\subset X^+$ will be called \emph{a Cerami sequence} for $\tJ$ at the level $c$ if $(1+\|u_n\|)\tJ'(u_n) \to 0$ and $\tJ(u_n) \to c$. In view of (I4), it is clear that if $(u_n)$ is a bounded Cerami sequence for $\tJ$, then $(m(u_n))\subset \cM$ is a bounded Cerami sequence for $\cJ$.

\begin{Th}[\!\!\cite{MSchSz}]\label{ThLink1}
	Suppose $\cJ \in \cC^1(X,\R)$ satisfies (I1)--(I7). Then:
	\begin{itemize}
		\item[(a)] $c_{\cM}\ge a>0$ and $\tJ$ has a Cerami sequence $(u_n)$ at the level $c_\cM$. 
		\item[(b)] If $\cJ$ satisfies also (I8), then $c_{\cM}=c_{\cN}:= \inf_\cN \cJ$ and $\cN$ is bounded away from $\wt X$ (hence closed in $X$).
	\end{itemize}
\end{Th}

For a topological group acting on $X$, denote \emph{the orbit of $u\in X$} by $G\ast u$, i.e., 
$$G\ast u:=\{gu: g\in G\}.$$
A set $A\subset X$ is called \emph{$G$-invariant} if $gA\subset A$ for all $g\in G$. $\cJ\colon X\to\R$ is called \emph{$G$-invariant} and $T\colon X\to X^*$ \emph{$G$-equivariant} if $\cJ(gu)=\cJ(u)$ and $T(gu)=gT(u)$ for all $g\in G$, $u\in X$. \\
\indent   
In order to deal with  multiplicity of critical points, assume that $G$ is a topological group such that
\begin{itemize}
	\item[(G)] $G$ acts on $X$ by isometries and discretely in the sense that for each $u\ne 0$, $(G*u)\setminus\{u\}$ is bounded away from $u$. Moreover, $\cJ$ is $G$-invariant and $X^+,\wt X$ are $G$-invariant.
\end{itemize}
Observe that $\cM$ is $G$-invariant and $m\colon X^+\to\cM$ is $G$-equivariant.  In our application to \eqref{eq} we have $G=\mathbb{Z}^3$ acting by translations, see Theorem \ref{ThMain}.\\
\indent 
Since all the nontrivial critical points $u$ of $\cJ$ are in $\cN$, it follows from Theorem \ref{ThLink1} that $\tJ(u)\ge a$ for all such $u$. \\
\indent 
We introduce the following variant of the {\em Cerami condition} between the levels $\alpha, \beta\in\R$.
\begin{itemize}
	\item[$(M)_\alpha^\beta$]
	\begin{itemize}
		\item[(a)] Let $\alpha\le\beta$. There exists $M_\alpha^\beta$ such that $\limsup_{n\to\infty}\|u_n\|\le M_\alpha^{\beta}$ for every $(u_n)\subset X^+$ satisfying $\alpha\le\liminf_{n\to\infty}\tJ(u_n)\le\limsup_{n\to\infty}\tJ(u_n)\leq\beta$ and \linebreak $(1+\|u_n\|)\tJ'(u_n)\to 0$.
		\item[(b)] Suppose in addition that the number of critical orbits in $\tJ_\alpha^\beta$ is finite. Then there exists $m_\alpha^\beta>0$ such that if $(u_n), (v_n)$ are two sequences as above and $\|u_n-v_n\|<m_\alpha^\beta$ for all $n$ large, then  $\liminf_{n\to\infty}\|u_n-v_n\|=0$.
	\end{itemize}
\end{itemize}

Note that if $\cJ$ is even, then $m$ is odd (hence $\tJ$ is even) and $\cM$ is symmetric,  i.e., $\cM=-\cM$. Note also that $(M)_\alpha^\beta$ is a condition on $\tJ$ and \emph{not} on $\cJ$. Our main multiplicity result reads as follows.
\begin{Th}[\!\!\cite{MSchSz}]\label{Th:CrticMulti}
		Suppose $\cJ \in \cC^1(X,\R)$ satisfies (I1)--(I8) and $\dim(X^+)=\infty$.\\
		(a) If  $(M)_0^{c_\cM+\eps}$ holds for some $\eps>0$, then either $c_{\cM}$ is attained by a critical point or there exists a sequence of critical values $c_n$ such that $c_n>c_{\cM}$ and $c_n\to c_{\cM}$ as $n\to\infty$.\\		
		 (b) If $(M)_0^{\beta}$ holds for every $\beta>0$ and $\cJ$ is even, 
	then $\cJ$ has infinitely many distinct critical orbits.
\end{Th}

\section{Properties of the functional $J$ for curl-curl}\label{sec:Jcurlcurl}

Recall our earlier assumption that (N1), (N2'), (N3) and (F1)--(F4) hold. We will check that assumptions (I1)--(I8) are satisfied and we want to apply Theorems \ref{ThLink1} and \ref{Th:CrticMulti}.\\
\indent Define the manifold
\begin{equation*}
\mathcal{M} := \{(v,w)\in\V\times\W:
\J'(v,w)[(0,\psi)]=0\,\hbox{ for any }\psi\in \W\}
\end{equation*}
and the Nehari-Pankov set for $\J$
\begin{eqnarray*}
\mathcal{N} &:=& \{(v,w)\in\V\times\W: u\neq 0,\; 
\J'(v,w)[(v,w)]=0, \\
&&\hbox{ and }\J'(v,w)[(0,\psi)]=0\,\hbox{ for any }\psi\in \W\} 
\ \subset\ \cM. \nonumber
\end{eqnarray*}
Observe that $u=v+w\in\cN_{\cE}$ if and only if $(v,w)\in\cN$ ($\cN_{\cE}$ is defined in \eqref{DefOfNehari1}). Moreover, $\mathcal{N}$ contains all nontrivial critical points of $\J$. In general  $\cN_{\cE}$, $\cN$, and $\cM$ are not $\cC^1$-manifolds.

\begin{Prop}\label{Propuv_N}
	If $(v,w)\in \V\times\W$ then
	$$\J(tv,tw+\psi)-\J'(v,w)\Big[\Big(\frac{t^2-1}{2}v,\frac{t^2-1}{2}w+t\psi\Big)\Big]\leq\J(v,w)
	$$
	for any $\psi\in\W$ and $t\geq 0$. 
\end{Prop}

\begin{proof}
	Let $(v,w)\in\V\times\W$, $\psi\in\W$, $t\geq 0$. We define
	$$D(t,\psi):=\J(t v,tw+\psi)-\J(v,w)-\J'(v,w)\Big[\Big(\frac{t^2-1}{2}v,\frac{t^2-1}{2}w+t\psi\Big)\Big]-\frac{1}{2}\int_{\R^3}V(x)|\psi|^2\,dx$$
and observe that
	\begin{eqnarray*}
		D(t,\psi)
		&=&\int_{\R^3}\langle f(x,v+w),\frac{t^2-1}{2}(v+w)+t\psi\rangle\, dx\\
		&&+\int_{\R^3}F(x,v+w)-F(x,t(v+w)+\psi)\,dx.
	\end{eqnarray*}
Now we can argue as in \cite[Proof of Proposition 4.1]{MSchSz} and show that 
$$\langle f(x,v+w),\frac{t^2-1}{2}(v+w)+t\psi\rangle+F(x,v+w)-F(x,t(v+w)+\psi) \le 0$$
for all $t\geq 0$, $\psi\in\R^3$ and a.e. $x\in\R^3$, where (F4) plays a crucial role.
Since $V(x)\leq 0$, we conclude.
\end{proof}

Consider $I:L^{\Phi}\to \R$ and  $\cI:L^{\Phi}\times \W\to\R$ given by
\begin{equation}\label{DefOfXi}
\cI(v,w) := I(v+w) := \int_{\R^3} -\frac12 V(x) |v+w|^2 + F(x,v+w) \, dx \quad \hbox{ for } (v,w) \in L^{\Phi}\times\W
\end{equation}
and recall that $I$ and $\cI$ are of class $\cC^1$. Under the assumptions of Theorem \ref{ThMain}, $I$ and $\cI$ are strictly convex. Moreover, the following property holds.

\begin{Lem}\label{LemConvWeakIpliesStrong}
	If $u_n\rightharpoonup u$ in $L^{\Phi}$ and
	$I(u_n)\to I(u)$,
	then $u_n\to u$ in $L^{\Phi}$.
\end{Lem}

\begin{proof}
We show that (up to a subsequence) $u_n(x)\to u(x)$ a.e. on $\R^3$.
Since $I(u_n)\to I(u)$, we have
\begin{equation*}
	\lim_{n\to\infty}\int_{\R^3}V(x)|u_n|^2\,dx=\int_{\R^3}V(x)|u|^2\,dx.
\end{equation*}
and
\begin{equation} \label{i3}
	\lim_{n\to\infty}\int_{\R^3}F(x,u_n)\,dx=\int_{\R^3}F(x,u)\,dx.
\end{equation}
If $(V)$ holds, then  passing to a subsequence 
$\lim_{n\to\infty}\int_{\R^3}V(x)|u_n-u(x)|^2\,dx=0$
and $u_n(x)\to u(x)$ a.e. on $\R^3$. If $V=0$, then by the strict convexity
we infer that for any $0<r\leq R$,
\begin{equation*}
	m_{r,R}:=\inf_{\substack{x,u_1,u_2\in\R^3\\ r\leq|u_1-u_2|,\\|u_1|,|u_2|\leq R} }\;
	\frac{1}{2}\bigl(F(x,u_1)+F(x,u_2)\bigr)-F\Big(x,\frac{u_1+u_2}{2}\Big)>0.
\end{equation*} 
Observe that by \eqref{i3} and convexity of $F$,
$$
0
\leq \limsup_{n\to\infty}\int_{\R^3}\frac12(F(x,u_n) + F(x,u))
- F\left(x,\frac{u_n+u}{2}\right)\,dx
\leq 0.
$$
Therefore, setting
$$
\Om_n:=\{x\in\R^3: |u_n-u|\geq r,\;|u_n|\leq R,\;
|u|\leq R\},
$$
there holds
$$
|\Om_n| m_{r,R}
\leq \int_{\R^3}\frac12(F(x,u_n) + F(x,u))
- F\left(x,\frac{u_n+u}{2}\right)\, dx,
$$
and thus $|\Om_n|\to 0$ as $n\to\infty$. Since $0<r\leq R$ are arbitrarily chosen, we deduce
$$
u_n\to u \quad \hbox{a.e.\ on }\R^3.
$$
In view of  a variant of the Brezis-Lieb result \cite[Lemma 4.3]{MSchSz},  we obtain
$$\int_{\R^3}F(x,u_n)\, dx-\int_{\R^3}F(x,u_n-u)\, dx\to\int_{\R^3}F(x,u)\, dx$$
and hence
$$\int_{\R^3}F(x,u_n-u)\, dx\to0.$$
By (F2) and Lemma \ref{AllProp} (iii) we get $|u_n-u|_\Phi\to 0$.
\end{proof}

\begin{Lem}\label{lem:supQ}
Let $(v_n,w_n)\subset \cV\times\cW$, $0\ne v\in\cV$, and $(t_n)\subset(0,\infty)$ such that $v_n\weakto v$ in $\cD$ and $t_n\to\infty$ as $n\to\infty$. Then
\[
\lim_{n\to\infty} \frac{1}{t_n^2} \int_{\R^3} \Phi(t_n|v_n+w_n|) \, dx = \infty.
\]
\end{Lem}
\begin{proof}
We argue as in proof of \cite[Proposition 4.4]{MSchSz}.
Take $R_0>0$ such that $v\neq 0$ in $B(0,R_0)$.
In view of (N3) we find $C>0$ such that
$$C\Phi(t)\geq t^2\quad\hbox{for }t\geq 1.$$
Then 
\begin{equation*}
	\int_{B(0,R)}|v_n+w_n|^2\,dx\leq C\int_{\R^3}\Phi(t_n|v_n+w_n|)/t_n^2\,dx+\int_{B(0,R)\cap\{|v_n+w_n|\leq 1\}}|v_n+w_n|^2\,dx
\end{equation*}
and the statement is true provided $v_n+w_n$ is unbounded in $L^2(B(0,R),\R^3)$ for some $R\geq R_0$. Now, suppose that $v_n+w_n$ is bounded in $L^2(B(0,R),\R^3)$ for any $R\geq R_0$.  
We may assume passing to a subsequence that $v_n\to v$ a.e. and $w_n\weakto w$ in $L^2_{loc}(\R^3,\R^3)$ for some $w$. Given $\eps>0$, let 
\begin{equation}\label{aetoinfI7}
	\Omega_n := \{x\in\R^3: |v_n(x)+w_n(x)|\ge\eps\}.
\end{equation}
We claim that there exists $\eps>0$ such that $\lim_{n\to\infty}|\Omega_n|>0$, possibly after passing to a subsequence. 
Arguing indirectly, suppose this limit is 0 for each $\eps$. Then $v_n+w_n\to 0$ in measure, so up to a subsequence $v_n+w_n\to 0$ a.e., hence $w_n\to -v$ a.e. and  $w_n\rh -v$ in $L^2_{loc}(\R^3,\R^3)$. Since $\curlop w_n=0$ in the distributional sense, the same is true of $v$. Thus there is $\xi\in H^1_{loc}(\R^3)$ such that $v=\nabla \xi$, see \cite[Lemma 1.1(i)]{Leinfelder}. As $\div(\nabla \xi)=\div v =0$, it follows that $\xi$, and therefore $v$, is harmonic. Recalling that $v\in\cD$, we obtain $v=0$. This is a contradiction.
Taking $\eps$ in \eqref{aetoinfI7} such that $\lim_{n\to\infty}|\Omega_n|>0$, we obtain
\[
\int_{\R^3} \frac{\Phi(t_n|v_n+w_n|)}{t_n^2} \, dx = \int_{\R^3} \frac{\Phi(t_n|v_n+w_n|)}{t_n^2|v_n+w_n|^2}|v_n+w_n|^2 \, dx \ge \eps^2 \int_{\Om_n} \frac{\Phi(t_n|v_n+w_n|)}{t_n^2|v_n+w_n|^2} \, dx \to \infty.\qedhere
\]
\end{proof}

\begin{Prop}\label{PropDefOfm(u)} Conditions  (I1)--(I8) are satisfied and
 there is a Cerami sequence $(v_n)\subset \cV$ for $\cJ\circ m$ at the level $c_{\cN}$, i.e.,   $\J\circ m(v_n)\to c_{\cN}$ and $(1+\|v_n\|_\cD)(\J\circ m)'(v_n)\to 0$ as $n\to\infty$, where
	$$c_{\cN}:=\inf_{(v,w)\in\mathcal{N}}\J(v,w)>0.$$
\end{Prop}
\begin{proof}
	Setting $X:=\V\times\W$,
	$X^+:=\V\times\{0\}$ and $\tX:=\{0\}\times\V$ we check assumptions (I1)--(I8) for the functional $\J\colon X\to\R$ 
	given by
	$$\J(v,w)=\frac{1}{2}\|v\|^2_\D-\cI(v,w)$$
	(cf. \eqref{eqJ} and \eqref{DefOfXi}). Recall 
	\[
	\|(v,w)\|:=\bigl(\|v\|_{\cD}^2+|w|_{\Phi}^2\bigr)^{\frac{1}{2}}, \quad \text{where }  \|v\|_{\cD} = |\nabla v|_2.
	\]
	Convexity and differentiability of $\cI$, (F2), and Lemma \ref{LemConvWeakIpliesStrong} yield:
	\begin{itemize}
		\item[(I1)] $\cI|_{\V\times\W}\in \cC^1(\V\times\W,\R)$ and $\cI(v,w)\geq \cI(0,0)=0$ for any $(v,w)\in \V\times\W$.
		\item[(I2)] If $v_n\to v$ in $\V$, $w_n\weakto w$ in $\W$,  then $\displaystyle\liminf_{n\to\infty}\cI(v_n,w_n)\geq \cI(v,w)$.
		\item[(I3)] If $v_n\to v$ in $\V$, $w_n\weakto w$ in $\W$ and $\cI(v_n,w_n)\to \cI(v,w)$, then $(v_n,w_n)\to (u,w)$.
	\end{itemize}
	Moreover,
	\begin{itemize}
		\item[(I6)] There exists $r>0$ such that $\inf_{\|v\|_\D=r}\J(v,0)>0$.
	\end{itemize}
Take $v\in\cV$ and
observe that 
\begin{eqnarray*}
	\cJ(v)&\geq& \frac12\int_{\R^3} |\curlop v|^2+V(x)|v|^2\,dx-c\int_{\R^3}|v|^6\,dx\\
	&\geq& \frac12\big(1- |V|_{\frac32}S^{-1}\big)\int_{\R^3} |\curlop v|^2\,dx -cS^{-3}\Big(\int_{\R^3} |\nabla v|^2\,dx\Big)^3\\
		&=& \frac12\big(1- |V|_{\frac32}S^{-1}\big)\int_{\R^3} |\curlop v|^2\,dx -cS^{-3}\Big(\int_{\R^3} |\curlop v|^2\,dx\Big)^3
\end{eqnarray*}
Hence, the inequality $|V|_{\frac32}<S$ implies that there is $r>0$ such that
(I6) is satisfied.
	It is easy to verify using (F2) and (v) of Lemma \ref{AllProp} that
	\begin{itemize}
		\item[(I4)] $\|v\|_\D+\cI(v,w)\to\infty$ as $\|(v,w)\|\to\infty$.
	\end{itemize}
	Hence also
	\begin{itemize}
	\item[(I5)] If $(v,w)\in\cM$, then $\cI(v,w)<\cI(v,w+\psi)$ for any $\psi\in \cW\setminus\{0\}$
	\end{itemize}
	holds by strict convexity of $u\mapsto \int_{\R^3}-V|u|^2+F(x,u)\,dx$.
	Next we  prove 
	\begin{itemize}
		\item[(I7)] $\cI(t_n(v_n,w_n))/t_n^2\to\infty$ if $t_n\to\infty$ and $v_n\to v$ for some $v\neq 0$ as $n\to\infty$.
	\end{itemize}
	Observe that from (F2)
\begin{equation}\label{eq:I7check}
\frac{\cI\bigl(t_n(v_n,w_n)\bigr)}{t_n^2} \ge \frac{1}{t_n^2} \int_{\R^3} F\bigl(x,t_n(v_n+w_n)\bigr) \, dx \ge c_2 \frac{1}{t_n^2} \int_{\R^3} \Phi(t_n|v_n+w_n|) \, dx,
\end{equation}
so the statement follows from Lemma \ref{lem:supQ}.
Finally, Proposition \ref{Propuv_N} and a simple computation show that
	\begin{itemize}
		\item[(I8)]
		$\frac{t^2-1}{2}\cI'(v,w)[(v,w)]+t\cI'(v,w)[(0,\psi)]+\cI(v,w)-\cI(tv,tw+\psi)\leq 0$
		for any $t\geq 0$, $v\in\V$ and $w,\psi\in\W$.
	\end{itemize}
Recall that \textcolor{red}{if} $(v_n)$ is a bounded Cerami sequence for $\tJ$, then $(m(u_n))\subset \cM$ is a bounded Cerami sequence for $\cJ$. Therefore applying Theorem \ref{ThLink1} we obtain the last conclusion.
\end{proof}

\begin{Lem}[\!\!\cite{MSchSz}]\label{LemDefofW}$\hbox{}$\\
	$(a)$ For any $v\in L^{\Phi}$ there is a unique $w(v)\in \cW$ such that
	\begin{equation*}
		\cI(v,w(v))=\inf_{w\in\W}\cI(v,w).
	\end{equation*}
	Moreover, $w\colon L^{\Phi}\to\W$  is continuous.\\
	$(b)$ $w$ maps bounded sets into bounded sets and $w(0)=0$.
\end{Lem}

\section{Proof of Theorem \ref{ThMain}}\label{sec:T1}

In this Section we assume that (N2) is satisfied. Since there is no compact embedding of $\V$ into $L^{\Phi}$, we cannot expect that the Palais--Smale or Cerami condition is satisfied. However, in view of  \cite[Lemma 1.5]{MederskiZeroMass} the following
variant of Lions' lemma holds.

\begin{Lem}\label{lem:Conv}
	Suppose that   $(v_n)\subset \D$ is bounded and for some $r>\sqrt{3}$
	\begin{equation}\label{eq:LionsCond11}
	\sup_{y\in \Z^3}\int_{B(y,r)}|v_n|^2\,dx\to 0 \quad\hbox{as } n\to\infty.
	\end{equation}
	Then  
	$$\int_{\R^3} \Phi(|v_n|)\, dx\to 0\quad\hbox{as } n\to\infty.$$
\end{Lem}

Let $m(v):=(v,w(v))\in\cM$ for $v\in\cV$. Then in view of Lemma \ref{LemDefofW} (a), $m\colon\cV\to\cM$ is continuous.
The following lemma implies that any Cerami sequence of $\J$ in $\cM$ and any Cerami sequence of $\cJ\circ m$
are bounded.

\begin{Lem}\label{LemCoercive}
Let $\beta>0$. There exists $M_\beta>0$ with the property that, if $(v_n)\subset \cV$ is such that $(\J\circ m)(v_n)\leq\beta$ and $(1+\|v_n\|)(\J\circ m)'(v_n)\to 0$ as $n\to\infty$, then  $\limsup_{n\to\infty} \|v_n\| \le M_\beta$.
\end{Lem}
\begin{proof}
If no finite bound $M_\beta$ exists, then we find a sequence $(v_n)\subset\cV$ such that 
$\|v_n\|\to\infty$ as $n\to\infty$, $(\J\circ m)(v_n)\le \beta$, and $(1+\|v_n\|)(\J\circ m)'(v_n)\to 0$ as $n\to\infty$. Since $w_n=w(v_n)$, $\|(v_n,w_n)\|\to\infty$ if and only if $\|v_n\|_\D\to\infty$.
Let $\bar{v}_n:=v_n/\|v_n\|_\D$ and $\bar{w}_n:=w_n/\|v_n\|_\D$. 
Assume 
$$\lim_{n\to\infty}\sup_{y\in\Z^3}\int_{B(y,r)}|\bar{v}_n|^2\,dx=0$$
for some fixed $r>\sqrt{3}$. 
By Lemma \ref{lem:Conv}, $\lim_{n\to\infty}\int_{\R^3}\Phi(|\bar{v}_n|)\,dx=0$, and arguing similarly as \cite{MSchSz}, we obtain a contradiction. More precisely, recalling $\cJ'(v_n,w_n)[(0,w_n)]=0$,
Proposition \ref{Propuv_N} with $t_n=s/\|v_n\|_\D$ and $\psi_n=-t_nw_n$ implies that for every $s>0$,
\begin{eqnarray*}
\beta &\ge&\limsup_{n\to\infty}\J(v_n,w_n) \\ &\ge&\limsup_{n\to\infty}\J(s\bar{v}_n,0)-\lim_{n\to\infty}\J'(v_n,w_n)\Big[\Big(\frac{t_n^2-1}{2}v_n,-\frac{t_n^2+1}{2}w_n\Big)\Big]=\limsup_{n\to\infty}\J(s\bar{v}_n,0)\\
&\overset{(F2)}{\ge}&\frac{s^2}{2}-\lim_{n\to\infty}c_1\int_{\R^3}\Phi(s|\bar v_n|)\,dx
 =\frac{s^2}{2}
\end{eqnarray*}
which is impossible. Hence $\liminf_{n\to\infty}\int_{B(y_n,r)}|\bar{v}_n|^2\,dx>0$ for some sequence $(y_n)\subset\Z^3$ and, up to a subsequence, we may assume that
$$\int_{B(0,r)}|\bar{v}_n(x+y_n)|^2\,dx\geq c>0$$
for some constant $c$. This implies that up to a subsequence, $\bar{v}_n(\cdot+y_n)\rightharpoonup\bar v\ne 0$ in $\D$, $\bar{v}_n(\cdot+y_n)\to\bar{v}$ in $L^2_{loc}(\R^3,\R^3)$, and $\bar{v}_n(\cdot+y_n)\to\bar v$ a.e. in $\R^3$ for some $\bar v\in\cD$.
By (F3),
\[
2\J(v_n,w_n) - \J'(v_n,w_n)[(v_n,w_n)] = \int_{\R^3}(\langle f(x,v_n+w_n), v_n+w_n\rangle - 2F(x,v_n+w_n))\,dx \ge 0,
\]
so $\cJ(v_n,w_n)$ is bounded below and
\[\begin{split}
\alpha & \le \frac{\J(v_n,w_n)}{\|v_n\|_\D^2}\le \frac{1}{2}\|\bar{v}_n\|_\D^2-c_2\int_{\R^3}\frac{\Phi(|v_n+w_n|)}{\|v_n\|_\cD^2}\,dx\\
& = \frac12 - c_2\int_{\R^3} \Phi\left(\|v_n\|_\cD\frac{|v_n(\cdot+y_n) + w_n(\cdot+y_n)|}{\|v_n\|_\cD}\right) \frac{1}{\|v_n\|_\cD^2} \, dx
\end{split}\]
for some constant $\alpha$ (cf. \eqref{eq:I7check} for the second inequality) and the integral on the right-hand side above tends to $\infty$ due to Lemma \ref{lem:supQ}, a contradiction.
\end{proof}

Now we show the weak-to-weak$^*$ convergence in $\cM$.

\begin{Prop} \label{prop}
	If $v_n\rh v$ in $\cD$, then $w(v_n)\rh w(v)$ in $\cW$ and, after passing to a subsequence, $w(v_n)\to w(v)$ a.e. in $\R^3$.
\end{Prop}

\begin{proof}
	Let us define $\tilde{F}(x,u):=-\frac12V(x)|u|^2+F(x,u)$ and $\tilde{f}(x,u):=-V(x)u+f(x,u)$ and observe that we can apply arguments of \cite[Proposition 5.2]{MSchSz} to $\tilde{F}$ and  $\tilde{f}$.
\end{proof}

In general $\J'$ is not (sequentially) weak-to-weak$^*$ continuous,
however we show the weak-to-weak$^*$ continuity of $\J'$ for sequences on the topological manifold $\mathcal{M}$. Obviously, the same regularity holds for $\cE'$ and $\cM_{\cE}$.

\begin{Cor}\label{CorJweaklycont}
	If $(v_n,w_n)\in\mathcal{M}$ and $(v_n,w_n)\rightharpoonup (v_0,w_0)$ in $\V\times\W$ then $\J'(v_n,w_n)\rightharpoonup \J'(v_0,w_0)$, i.e.
	$$\J'(v_n,w_n)[(\phi,\psi)]\to \J'(v_0,w_0)[(\phi,\psi)]$$
	for any $(\phi,\psi)\in\V\times\W$.
\end{Cor}

\begin{proof}
From Lemma \ref{LemDefofW}$(a)$ we get $w_n=w(v_n)$, while in view of Proposition \ref{prop} we may assume $v_n+w_n\to v_0+w_0$ a.e. in $\R^3$ (where $w_0=w(v_0)$).
For any $(\phi,\psi)\in\V\times\W$ we have
\begin{eqnarray*}
	\J'(v_n,w_n)[(\phi,\psi)]-\J'(v_0,w_0)[(\phi,\psi)]&=&\int_{\R^3} \langle \nabla v_n-\nabla v_0,\nabla \phi\rangle \ ,dx\\
	&&+\int_{\R^3} V(x)\langle v_n+w_n-v_0-w_0,\phi+\psi\rangle \, dx\\
	&&-\int_{\R^3} \langle f(x,v_n+w_n)-f(x,v_0+w_0),\phi+\psi\rangle \, dx.
\end{eqnarray*}
For $\Om\subset\R^3$ measurable, in view of Lemma \ref{AllProp} (iii) we have
\[
\int_{\Om} |\langle f(x,v_n+w_n)-f(x,u_0+w_0),\phi+\psi\rangle|\,dx \le |f(x,v_n+w_n)-f(x,u_0+w_0)|_\Psi |\phi+\psi|_{L^\Phi(\Om)}
\]
and using (F2) and Lemma \ref{lem:phd}, $(|f(x,v_n+w_n)-f(x,u_0+w_0)|_\Psi )$ is bounded. Since the norm $|\cdot|_\Phi$ is absolutely continuous (cf. \cite[Definition III.4.2, Corollary III.4.5 and Theorem III.4.14]{RaoRen}) and $\int_{\Om}\Phi(|\phi+\psi|)\,dx\to 0$ as $\Om=\R^3\setminus B(0,n)$ and $n\to\infty$, we infer that 
$\langle f(x,v_n+w_n)-f(x,u_0+w_0),\phi+\psi\rangle$ is uniformly integrable and tight; likewise for $V\langle v_n+w_n-v_0-w_0,\phi+\psi\rangle$, but using Lemma \ref{lem:V} (iii) instead.
In view of the Vitali convergence theorem we obtain
\[
\J'(v_n,w_n)[(\phi,\psi)]-\J'(v_0,w_0)[(\phi,\psi)]\to 0.\qedhere
\]
\end{proof}

\begin{altproof}{Theorem \ref{ThMain} $(a)$}\mbox{}\\ 
	The existence of a Cerami sequence $(v_n,w_n)\subset \cM$ at the level $c_{\cN}$ follows from Proposition \ref{PropDefOfm(u)}, and this sequence is bounded by 
	Lemma \ref{LemCoercive}.\\ 
	\indent Suppose that $V=0$. 
	 If $|v_n|_\Phi\to 0$, then
	by (F2) and Lemma \ref{LemCoercive}, we have
	\[
	\begin{split}
		\|v_n\|_{\mathcal{D}}^2 = & \, J'(m(v_n))[(v_n,0)]
		+\int_{\R}\langle f(x,m(v_n)),v_n\rangle\,dx\le o(1)+\int_{\R}|f(x,m(v_n))||v_n|\,dx\\
		\le & \, o(1)+c_1\int_{\R}\Phi'(|m(v_n)|)|v_n|\,dx \le o(1)+ c_1|\Phi'(|m(v_n)|)|_\Psi|v_n|_{\Phi} \ \to\ 0
	\end{split}
	\]
	which gives $v_n\to 0$. This is impossible because $\J(m(v_n))\to c_\cN>0$. Hence by Lemma \ref{lem:Conv}
$
		\sup_{y\in \Z^3}\int_{B(y,r)}|v_n|^2\,dx 
$
is bounded away from $0$ and we find translations $(y_n)\subset\Z$ such that
\begin{equation}\label{eqVproof}
\int_{B(0,R)}|v_n(\cdot +y_n)|^2\,dx\ge \eps
\end{equation}
for some $\eps>0$. We find  $v\in\cV\setminus\{0\}$ such that
$(v_n(\cdot +y_n), w_n(\cdot +y_n)) \weakto (v,w)$ and $(v_n(\cdot +y_n), w_n(\cdot +y_n)) \to (v,w)$  a.e. in $\R^3$ along a subsequence.
Then by Fatou's lemma and (F3),
	\begin{eqnarray*}
		c_{\cN}&=&\lim_{n\to\infty}\cJ(v_n,w_n)=\lim_{n\to\infty}\Big(\cJ(v_n(\cdot +y_n),w_n(\cdot +y_n))\\
		&&-\frac12\cJ'(v_n(\cdot +y_n),w_n(\cdot +y_n))[(v_n(\cdot +y_n),w_n(\cdot +y_n))]\Big) \\
		&\geq&\cJ(v,w)-\frac12\cJ'(v,w)[(v,w)] =\cJ(v,w).
	\end{eqnarray*}
	Since $(v,w)\in\cN$,  $\J(v,w)=c_{\cN}$ and $u=v+w$ solves \eqref{eq}.\\ 
\indent 
	Now suppose that $(V)$ holds. Let $\cJ_0$, $\cM_0$, $\cN_0$, $c_{\cN_0}$, and $m_0$ denote the energy functional, the manifolds $\cM$ and $\cN$, the ground state energy, and the homeomorphism $\cV\to\cM$ in the case $V=0$. Let $u_0$ be the ground state obtained above. Observe that
	\begin{equation}\label{eq:CNN_0}
c_{\cN_0}=\cJ_0(u_0)\geq \cJ_0(m(u_0))>\cJ(m(u_0))\geq c_{\cN}.
	\end{equation}
	Suppose that passing to a subsequence $v_n\weakto 0$ and a.e. on $\R^3$. Then $m_0(v_n)\weakto 0$ and for every $R>0$ there holds
	$$ \int_{\R^3} V(x)|m_0(v_n)|^2 \, dx = \int_{B(0,R)} V(x)|m_0(v_n)|^2 \, dx + \int_{\R^3\setminus B(0,R)} V(x)|m_0(v_n)|^2 \, dx
	$$
	with $\lim_n \int_{B(0,R)} V(x)|m_0(v_n)|^2 \, dx = 0$ because, from \cite[Lemma 5.1]{MSchSz} and (N3), $\cD$ embeds compactly in $L^\Phi_{loc}(\R^3,\R^3) \hookrightarrow L^2_{loc}(\R^3,\R^3)$. We want to prove that $\int_{\R^3\setminus B(0,R)} V(x)|m_0(v_n)|^2 \, dx \to 0$ uniformly in $n$ as $R\to\infty$. From Lemma \ref{AllProp} (iii) and since $m_0(v_n)$ is bounded in $L^{\Phi}$, there exists $C>0$ such that for every $n$
	\[
	\int_{\R^3\setminus B(0,R)} V(x)|m_0(v_n)|^2 \, dx \le C|Vm_0(v_n)\chi_{\R^3\setminus B(0,R)}|_{\Psi}.
	\]
	Again from Lemma \ref{AllProp}, it suffices to prove that
	\[
	\lim_{R\to\infty} \int_{\R^3} \Psi\bigl(|V(x)m_0(v_n)|\chi_{\R^3\setminus B(0,R)}\bigr) \, dx = 0 \quad \text{uniformly in } n.
	\]
	Arguing as in Lemma \ref{lem:V}, there exists $C'>0$ such that for every $n$
	\[
	\int_{\R^3} \Psi\bigl(|V(x)m_0(v_n)|\chi_{\R^3\setminus B(0,R)}\bigr) \, dx \le C'|V\chi_{\R^3\setminus B(0,R)}|_{(\Phi\circ\Psi^{-1})^*}
	\]
	and, since $(\Phi\circ\Psi^{-1})^*\in\Delta_2$,
	\[
	\lim_{R\to\infty} |V\chi_{\R^3\setminus B(0,R)}|_{(\Phi\circ\Psi^{-1})^*} = 0
	\]
	from the dominated convergence theorem. Hence
	$$\cJ(u_n)\geq \cJ(m_0(u_n))= \cJ_0(m_0(u_n))-\frac12 \int_{\R^N}V(x)|m_0(u_n)|^2\,dx
	\geq c_{\cN_0}+o(1),$$
	which contradicts \eqref{eq:CNN_0}. Therefore $(v_n,w_n)\weakto (v,w)$ and a.e. on $\R^3$ along a subsequence and $v\neq 0$. Arguing as above we obtain that $\cJ(v,w)=c_{\cN}$ and $u=v+w$ solves \eqref{eq}.
\end{altproof}

Now we sketch the approach of \cite{MSchSz} and we apply Theorem \ref{Th:CrticMulti} $(b)$.
Recall that the group $G:=\Z^3$ acts isometrically by translations on $X=\cV\times\cW$ and $\cJ$ is $\Z^3$-invariant.
Let
$$\cK:=\big\{v\in \cV: (\cJ\circ m)'(u)=0\big\}$$
and suppose that $\cK$ consists of a finite number of distinct orbits. It is clear that $\Z^3$ acts discretely and hence satisfies the condition (G) in Section \ref{sec:criticaslpoitth}.
Then, it is easy to see that
$$\kappa:= \inf\big\{\|v-v'\|_{\D}:\J'\bigl(m(v)\bigr) = \J'\bigl(m(v')\bigr) = 0, v\ne v'\big\}>0.$$
We recall \cite[Lemma 6.1]{MSchSz}.

\begin{Lem}\label{Discreteness}
Let $\beta\ge c_{\cN}$ and suppose that $\cK$ has a finite number of distinct orbits.  
If $(u_n),(v_n)\subset\V$ are two Cerami sequences for $\J\circ m$ such that $0\le\liminf_{n\to\infty}\J\bigl(m(u_n)\bigr)\le \limsup_{n\to\infty}\J\bigl(m(u_n)\bigr)\le\beta$, $0\le\liminf_{n\to\infty}\J\bigl(m(v_n)\bigr)\le \limsup_{n\to\infty}\J\bigl(m(v_n)\bigr)\le\beta$ and $\liminf_{n\to\infty}\|u_n-v_n\|_{\D}< \kappa$, then $\lim_{n\to\infty}\|u_n-v_n\|_{\D}=0$.
\end{Lem}

\begin{altproof}{Theorem \ref{ThMain} $(b)$}\mbox{}\\ 
In order to complete the proof we use directly Theorem \ref{Th:CrticMulti} $(b)$. That (I1)--(I8) are satisfied and $(M)_0^\beta$ holds for all $\beta>0$ follows from Proposition \ref{PropDefOfm(u)}, Lemma \ref{LemCoercive}, and Lemma \ref{Discreteness}.
\end{altproof}

\section{Proof of Theorem \ref{Th:main1}}\label{sec:T2}

Recall that $V=0$, $f(x,u):=|u|^4u$, and $\Phi(t)=\frac16|t|^6$.
In view of Proposition \ref{PropDefOfm(u)} there is a Cerami sequence $(v_n,w_n)\subset \cM$ at the level $c_{\cN}$, so in particular, $\cJ'\bigl(m(v_n)\bigr)\to 0$ as $n\to\infty$. \\
\indent For $s>0$, $y\in\R^3$, and $u\colon\R^3\to\R^3$ we denote $T_{s,y}(u):= s^{1/2}u(s\cdot +y))$. 
The following lemma is a special case of \cite[Theorem 1]{Solimini}, see also \cite[Lemma 5.3]{Tintarev}.

\begin{Lem}\label{lem:Solimini}
	Suppose that $(v_n)\subset \cD$ is bounded. Then $v_n\to 0$ in $L^6(\R^3,\R^3)$ if and only if $T_{s_n,y_n}(v_n)\weakto 0$ in 
$\cD$  for all $(s_n)\subset\R^+$ and $(y_n)\subset \R^3$.
\end{Lem}

We recall the following properties \cite{MSz}:
\begin{Lem} \label{isom}
	$T_{s,y}$ is an isometric isomorphism of $W^6(\curl;\R^3)$ which leaves the functional $\cE$ and the subspaces $\cV, \cW$ invariant. In particular, $w(T_{s,y}u)=T_{s,y}w(u)$.
\end{Lem}
Moreover, in view of \cite[Theorem 3.1]{MSz} the topological manifold 
\begin{equation*}
	\cM=\Big\{u\in W^6(\curl;\R^3): \div(|u|^4u)=0\Big\}
\end{equation*}
is locally compactly embedded in $L^p(\R^3,\R^3)$ for $1\leq p<6$.

\begin{altproof}{Theorem \ref{Th:main1}} 
We refer the reader to \cite{MSz}, where the inequality $S_\curl>S$ has been proved and $c_{\cN}=\inf_{\cN_{\cE}}\cE=\frac13S_{\curl}^{3/2}$ is obtained. We demonstrate  how the previous results can be applied to show the existence of ground state solutions in the critical case.\\
\indent 	Observe that
	\begin{equation} \label{ac}
		\cJ\bigl(m(v_n)\bigr) = \cJ\bigl(m(v_n)\bigr) -\frac16\cJ'\bigl(m(v_n)\bigr)[m(v_n)] = \frac13|\curlop v_n|^2_2 = \frac13|\nabla v_n|^2_2
	\end{equation}
	and $|\nabla\cdot|_2$ is an equivalent norm in $\cV$, hence $(v_n)$ is bounded. 
	In view of Lemma \ref{LemDefofW}, $\bigl(m(v_n)\bigr)$ is bounded and since
	\begin{equation} \label{ca}
		c_{\cN}\leq \cJ\bigl(m(v_n)\bigr) = \cJ(m(v_n)) -\frac12\cJ'\bigl(m(v_n)\bigr)[m(v_n)]  = \frac13|m(v_n)|^6_6,
	\end{equation} 
	 $|v_n|_6$ is bounded away from $0$.	  
	Therefore, passing to a subsequence and using Lemma \ref{lem:Solimini}, $\tv_n:=T_{s_n,y_n}(v_n)\weakto v_0$ for some $v_0\neq 0$, $(s_n)\subset\R^+$, and $(y_n)\subset \R^3$. Taking subsequences again we also have that $\tv_n\to v_0$ a.e. in $\R^3$ and in view of  the compact embedding of $\cM$ into $L^2_{loc}(\R^3,\R^3)$, $w(\tv_n)\weakto w_0$ and $w(\tv_n)\to w_0$ a.e. in $\R^3$. From Vitali's convergence theorem we obtain $I'(v_0+w_0)[w]=0$ for all $w\in\cW$ and so, since $I$ is strictly convex, $w_0=w(v_0)$.
	We set $u:=v_0+w(v_0)$ and by Lemma \ref{isom} we may assume without loss of generality that $s_n=1$ and $y_n=0$. So if $z\in W_0^6(\curl;\R^3)$, then using weak and a.e. convergence,
	\[
	\cJ'\bigl(m(v_n)\bigr)\vp= \int_{\R^3}\langle\curlop v_n, \curlop \vp\rangle\,dx -\int_{\R^3}\langle|m(v_n)|^4m(v_n),\vp\rangle\,dx \to J'(u)\vp
	\]
	for any $\vp\in\cC_0^{\infty}(\R^3,\R^3)$.
	Thus  $u$ is a solution to \eqref{eq}. To show it is a ground state, we note that using Fatou's lemma,
	\begin{eqnarray*}
		c_{\cN}& = & \cJ\bigl(m(v_n)\bigr)+o(1) = \cJ\bigl(m(v_n)\bigr)-\frac12\cJ'\bigl(m(v_n)\bigr)[m(v_n)]+o(1) = \frac13|m(v_n)|^6_6 + o(1) \\
		& \ge & \frac13|u|^6_6 + o(1) = \cJ(u)-\frac12\cJ'(u)[u]+o(1) = \cJ(u)+o(1).
	\end{eqnarray*}
	Hence $\cJ(u)\le c_{\cN}$ and as a solution, $u\in\cN$, and (b) is proved.
\end{altproof}

{\bf Acknowledgements.}
The authors were partly supported by the National Science Centre, Poland (Grant No. 2017/26/E/ST1/00817). J.~S. was partly supported also by the Grant No. 2020/37/N/ST1/00795.

\end{document}